\documentclass[ECP]{ejpecp} 

\usepackage[utf8]{inputenc}
\usepackage{enumerate}
\usepackage{bbm}
\usepackage{bigints}
\usepackage{amsthm}
\usepackage{amssymb}
\usepackage{color}
\usepackage{tikz}

\theoremstyle{definition}

\theoremstyle{remark}

\SHORTTITLE{Polynomial convergence to the Yaglom limit for Brownian motion with drift} 

\TITLE{Polynomial rate of convergence to the Yaglom limit for Brownian motion with drift} 

\AUTHORS{%
  William Oçafrain\footnote{Institut de Mathématiques, Université de Neuchâtel, Rue Emile-Argand, Neuchâtel,
Suisse-2000.
    \EMAIL{william.ocafrain@unine.ch}}
}


\AMSSUBJ{60B10;60J65;37A25} 

\SUBMITTED{} 
\ACCEPTED{} 

\VOLUME{}
\YEAR{}
\PAPERNUM{}
\DOI{}

\ABSTRACT{This paper deals with the rate of convergence in $1$-Wasserstein distance of the marginal law of a Brownian motion with drift conditioned not to have reached $0$ towards the Yaglom limit of the process. In particular it is shown that, for a wide class of initial measures including probability measures with compact support, the Wasserstein distance decays asymptotically as $1/t$. Likewise, this speed of convergence is recovered for the convergence of marginal laws conditioned not to be absorbed up to a horizon time towards the Bessel-$3$ process, when the horizon time tends to infinity.}

\def\E{\mathbb{E}}
\def\P{\mathbb{P}}
\def\R{\mathbb{R}}
\def\Q{\mathbb{Q}}

\def\1{\mathbbm{1}}

\def\cL{\mathcal{L}}

\def\cM{\mathcal{M}}
\def\cB{\mathcal{B}}

\def\cW{\mathcal{W}}
\def\Yaglom{\text{Yaglom}}
\def\Lip{\text{Lip}}
\def\Kolm{\text{Kolm}}

\begin{document}



\section*{Notation}

For a general set $F$, we will use the following notation :

\begin{itemize}
    \item $\cM_1(F)$ : Set of the probability measures defined on $F$.
    \item $\cB(F)$ : Set of the measurable bounded functions defined on $F$.
    \item $\cB_1(F)$ : Set of the measurable bounded functions defined on $F$ such that $||f||_\infty \leq 1$.
    \item For any $\mu \in \cM_1(F)$ and $f \in \cB(F)$, 
    $$\mu(f) := \int_F f(x) \mu(dx).$$
\item For any family of probability measure $(\mu_t)_{t \geq 0}$ and $\mu \in \cM_1(F)$, the notation 
$$\mu_t \underset{t \to \infty}{\overset{\cL}{\longrightarrow}} \mu$$
refers to the weak convergence of $(\mu_t)_{t \geq 0}$ towards $\mu$, that is : for any continuous and bounded measurable function $f$,
$$\mu_t(f) \underset{t \to \infty}{\longrightarrow} \mu(f).$$ 
 
\end{itemize}
For any $x, L \in \R_+ \times (0,+\infty)$, denote by $\Lip_{x,L}(\R_+)$ the set of the $L$-Liptschitz functions defined on $\R_+$ such that $f(0) = x$.
\section{Introduction}

Denote by $(X_t)_{t \geq 0}$ a Brownian motion with constant drift, that is defined as
\begin{equation}
    \label{brownian-drift}
    X_t := X_0 + B_t - rt,~~~~\forall t \geq 0,
\end{equation}
where $X_0$ is a random variable on $\R$, $(B_t)_{t \geq 0}$ is a one-dimensional Brownian motion starting from $0$ independent from $X_0$ and $r > 0$.  Denote by $(\P_x)_{x \in \R_+}$ a family of probability measures such that, for any $x > 0$, $\P_x(X_0 = x) = 1$. Then, for a given $\mu \in \cM_1(\R_+\setminus \{0\})$,  the probability measure $\P_\mu := \int_{(0,+\infty)} \P_x \mu(dx)$ is such that $\P_\mu(X_0 \in \cdot) = \mu$. Denote by $\E_x$ and $\E_\mu$ the associated expectations.  

Denote by $\tau_0$ the hitting time of $(X_t)_{t \geq 0}$ at $0$, i.e.
$$\tau_0 := \inf\{t \geq 0 : X_t = 0\}.$$
This paper will cope with the \textit{quasi-stationarity} for the process $(X_t)_{t \geq 0}$, that is the study of the asymptotic behavior of the Markov process $(X_t)_{t \geq 0}$ conditioned not to reach $0$. The main notion of this theory is the \textit{quasi-stationary distribution} (QSD), which is a probability measure $\alpha$ on $(0,+\infty)$ such that,  for any $t \geq 0$,
\begin{equation*}
    \label{qsd}
    \alpha = \P_{\alpha}(X_t \in \cdot | \tau_0 > t),
\end{equation*}
where, for any probability measure $\mu$ supported $(0,+\infty)$ and $t \geq 0$,
$$\P_\mu(X_t \in \cdot | \tau_0 > t) := \frac{\P_\mu(X_t \in \cdot, \tau_0 > t)}{\P_\mu(\tau_0 > t)}.$$
It is well known (see for example \cite{CMSM} or \cite{MV2012}) that to be a QSD is equivalent to the following property: there exists $\mu \in \cM_1((0,+\infty))$ such that
$$\P_\mu(X_t \in \cdot | \tau_0 > t) \underset{t \to \infty}{\overset{\cL}{\longrightarrow}} \alpha.$$
This paper will more specifically deal with the weak convergence of $\P_\mu(X_t \in \cdot | \tau_0 > t)$ towards the so-called \textit{Yaglom limit}, denoted by $\alpha_\Yaglom$, which is defined as the unique QSD such that, for any $x > 0$,
\begin{equation} \label{qld}\P_x(X_t \in \cdot | \tau_0 > t) \underset{t \to \infty}{\overset{\cL}{\longrightarrow}} \alpha_\Yaglom.\end{equation}
It is well known that such a probability measure exists for the Brownian motion with drift (see further in Section 2). The goal of this paper is more precisely to study the speed of convergence of the conditional probability measure $\P_\mu(X_t \in \cdot | \tau_0 > t)$,  for some initial measures $\mu$, towards the Yaglom limit when $t$ goes to infinity. 

In order to quantify the weak convergence \eqref{qld}, it is possible to use several distances on $\cM_1((0,+\infty))$. One of them is the \textit{total variation distance}, defined as follows : 
\begin{equation}
    \label{tv}
    ||\mu-\nu||_{TV} := \sup_{f \in \cB_1((0,+\infty))} |\mu(f)-\nu(f)|,~~~~\forall \mu, \nu \in \cM_1((0,+\infty)).
\end{equation}
In particular, the convergence towards $0$ of the total variation distance between the conditional probability $\P_\mu(X_t \in \cdot | \tau_0 > t)$ and $\alpha_\Yaglom$ implies the weak convergence \eqref{qld}. 

Another distance which can be used to quantify weak convergences is the \textit{$1$-Wasserstein distance}, defined as 
$$\cW_1(\mu,\nu) := \sup_{f \in \Lip_{1,1}(\R_+)} |\mu(f)-\nu(f)|.$$
For this distance, the decay towards $0$ implies a weak convergence and the convergence of the first moment.

For the most of the absorbed Markov processes, it is usually expected that the distance $||\P_\mu(X_t \in \cdot | \tau_0 > t) - \alpha_\Yaglom||_{TV}$ decreases exponentially fast. Especially, it is well known (see for example \cite{CV2014,CV2017c,V2018}) that some conditions, based on Doeblin-type condition or Lyapunov functions, entail an exponential decay of the total variation distance. However, in our case, the Brownian motion with drift does not satisfy such conditions.

 It even seems that an exponential speed could be too fast for some initial laws and that the expected rate of convergence is rather $1/t$. In particular, it was shown by Polak and Rolski in \cite{PR2012} that the $L^1$ distance between the density function of $\P_x[X_t \in \cdot | \tau_0 > t]$ and the one of the Yaglom limit is equivalent, when $t$ goes to infinity, to $c/t$, where $c > 0$ is a constant independent on the state $x$. In \cite{PV2017}, Palmowski and Vlasiou showed that, for a Lévy process satisfying some assumptions (see (SN) and (SP) in \cite{PV2017}) and taking, as initial measure, the invariant measure of the associated reflected process, the difference in absolute value between the density function of the conditional probability and the one of the Yaglom limit decreases also as $1/t$ when $t$ goes to infinity.  

The aim of this note is to recover this speed of convergence considering a large class of initial measures, improving therefore the result of Polak and Rolski, and using the Wasserstein distance $\cW_1$ to quantify the convergence. In particular, it will be shown that, if the initial measure $\mu$ has a compact support, there exists $c_\mu, C_\mu \in (0,\infty)$ such that one has asymptotically 
\begin{equation}
    \label{encadrement-intro}
    \frac{c_\mu}{t} \leq \cW_1(\P_\mu(X_t \in \cdot | \tau_0 > t),\alpha_\Yaglom)  \leq \frac{C_\mu}{t},
\end{equation}
The asymptotical inequalities \eqref{encadrement-intro} hold actually for a wider class of initial measures which will be spelled out later in Theorem \ref{theorem}, and the same result can be also stated replacing the Wasserstein distance by the total variation distance. 

The Section \ref{qsd-Bmwd} will begin by giving some useful and well-known generalities on the quasi-stationarity for the Brownian motion with drift absorbed at $0$. Then, the results of this note will be more precisely presented in Section \ref{wasserstein}, one of which is an important lemma on the asymptotic property of the Bessel-3 process. Finally, this note ends with the polynomial convergence of the conditional law $\P_\mu[X_{s} \in \cdot | \tau_0 > t]$ towards the marginal law at time $s$ of a Bessel-3 process, when $t$ goes to infinity.  

\section{Preliminaries on the quasi-stationarity for a Brownian motion with drift}
\label{qsd-Bmwd}
The quasi-stationarity for the Brownian motion with drift absorbed at $0$ has been studied by Martinez and San Martin in \cite{MSM1994}. In their paper, the authors showed that there exists an infinity of quasi-stationary distributions, one of which is the Yaglom limit $\alpha_{\Yaglom}$. Moreover, the density function of the Yaglom limit is explicitly given :
$$\alpha_{\Yaglom}(dx) = r^2 xe^{-rx}dx.$$
In another paper (\cite{MPSM98}), Martinez, Picco and San Martin are interested in the domain of attraction of the Yaglom limit, that is the set of the initial laws for which the convergence \eqref{qld} holds. In particular, they showed that, when the initial law $\mu$ admits a density function $\rho$ with respect to the Lebesgue measure, the conditional probability measure $\P_\mu(X_t \in \cdot | \tau_0 > t)$ converges to $\alpha_\Yaglom$ when
\begin{equation}
\label{liminf}
\liminf_{x \to \infty} - \frac{1}{x} \log (\rho(x)) \geq r.\end{equation}
Since $\alpha_\Yaglom$ is a quasi-stationary distribution, it is well known (see \cite{CMSM,MV2012}) that there exists $\lambda_0 > 0$ such that, for any $t \geq 0$,
$$\P_{\alpha_\Yaglom}(\tau_0 > t) = e^{-\lambda_0 t}.$$
For a Brownian motion with drift $r$, one has $$\lambda_0 = \frac{r^2}{2}.$$
Moreover, $\lambda_0$ is an eigenvalue for the infinitesimal generator of $(X_t)_{t \geq 0}$, which is
$$Lf(x) := \frac{1}{2}f''(x) - rf'(x),$$
and one can associate to $\lambda_0$ an eigenfunction $\eta$, which is unique up to a multiplicative constant and proportional to the function $x \mapsto x e^{rx}$. For example, one can choose
$$\eta(x) = \frac{1}{r^2} x e^{rx}.$$
From these definitions, the so-called $Q$-process can be defined as the Markov process whose the semi-group $(Q_t)_{t \geq 0}$ is defined by
\begin{equation}
\label{q-proc}
Q_tf(x) := \frac{e^{\lambda_0 t}}{\eta(x)} \E_x(\eta(X_t)f(X_t)\1_{\tau_0 > t}),~~~~\forall f \text{ measurable}, \forall t \geq 0.\end{equation}
For any positive measure $\mu$ supported on $(0,+\infty)$, one uses the notation
$$\mu Q_t f := \int_{(0,+\infty)} Q_tf(x) \mu(dx).$$
This $Q$-process is actually obtained from a Doob-transform of the sub-Markovian semi-group $P_tf(x) := \E_x(f(X_t)\1_{\tau_0 > t})$. It corresponds to the process conditioned "never" to be absorbed, in the following sense : the family of probability measure $(\Q_x)_{x > 0}$ defined as
$$\Q_x(\Gamma) = \lim_{T \to \infty} \P_x(\Gamma | \tau_0 > T),~~~~\forall ~t \geq 0, \forall~ \Gamma \in \sigma(X_s, 0 \leq s \leq t)$$
is well-defined and, for any $t \geq 0$ and any $f$ measurable,
$$Q_tf(x) = \E^\Q_x(f(X_t)),$$
where $\E^\Q_x$ is the expectation associated to $\Q_x$. 

Denote by $(Y_t)_{t \geq 0}$ the $Q$-process (i.e. $Q_tf(x) = \E_x(f(Y_t))$). Then $(Y_t)_{t \geq 0}$ is actually a Bessel-3 process, which is a diffusion process following 
$$dY_t = dB_t + \frac{1}{Y_t}dt.$$
Note that the $Q$-process does not depend on the drift $r > 0$ and one gets an explicit formula for the density function of $\P_x(Y_t \in \cdot)$ (for any $x > 0$ and $t \geq 0$), which is 
\begin{equation}
    \label{density}
    y \mapsto \frac{2}{\sqrt{2 \pi t}} \frac{y}{x} \sinh\left(\frac{xy}{t}\right) e^{-\frac{x^2+y^2}{2t}}.
\end{equation}
Moreover, the measure $\gamma(dx) = x^2dx$ is an invariant measure for the Bessel-3 process.  

\section{Polynomial convergence in Wasserstein distance}
\label{wasserstein}

The main result of this paper is now clearly stated :
\begin{theorem}
\label{theorem}
If the initial law $\mu \in \cM_1((0,+\infty))$ satisfies
\begin{equation}
\label{integrability}
\int_0^\infty x^3 e^{rx}\mu(dx) < +\infty,\end{equation} 
then 
\begin{multline}
0 < \liminf_{t \to \infty} ~t \times \cW_1(\P_\mu[X_t \in \cdot | \tau_0 > t],\alpha_\Yaglom)\\ \leq  \limsup_{t \to \infty}~ t \times \cW_1(\P_\mu[X_t \in \cdot | \tau_0 > t],\alpha_\Yaglom) < +\infty.\label{result}\end{multline}
\end{theorem}
Before proving the theorem, let us have a few remarks about its statement :
\begin{remark}In this paper, we will only focus on the Wasserstein distance, but it is also possible to obtain the same result \eqref{result} for others distances. In particular, one gets the same statement for Theorem \ref{theorem} taking the total variation distance instead, as defined in \eqref{tv}, or also the Kolmogorov distance defined as follows :
$$d_{\Kolm}(\mu,\nu) := \sup_{x \in \R} |\mu((-\infty,x]) - \nu((-\infty,x])|,~~~~\forall \mu, \nu \in \cM_1((0,+\infty)).$$
In the same way, the convergence in Kolmogorov distance implies the weak convergence of measures, but in a weaker way than the total variation distance or the Wasserstein distance. 
\end{remark}
\begin{remark}Concerning the domain of attraction of $\alpha_\Yaglom$, the assumption on the integrability of the initial measure \eqref{integrability} is slightly stronger than the assumption \eqref{liminf} written previously. As a matter of fact, \eqref{integrability} holds when the density function $\rho$ of the initial measure satisfies
\begin{equation*} \label{ex}\lim_{x \to \infty} -\frac{1}{x} \log \rho(x) > r, \end{equation*}
but does not hold when one has equality instead. Remark in particular that, taking $\mu = \alpha_\Yaglom$, \eqref{integrability} is not satisfied, as well as \eqref{result}. In a general way, the speed of convergence of $t \mapsto \cW_1(\P_\mu(X_t \in \cdot | \tau_0 > t), \alpha_\Yaglom)$ when $\lim_{x \to \infty} -\frac{1}{x} \log \rho(x) = r$ remains an open question. \end{remark}
\begin{remark}
As written in the introduction, this speed of convergence was already found out by Polak and Rolski in \cite{PR2012}. More precisely, the authors showed that there exists $c > 0$ such that, for any $x > 0$,
$$\lim_{t \to \infty} t \times \int_0^\infty |\P_x[X_t \in \cdot | \tau_0 > t] - \alpha_\Yaglom|(dy) = c.$$
In addition to the question concerning the choice of the distance, the main question is whether this result holds for others initial laws than Dirac measures. The proof of Polak and Rolski relies on the asymptotic expansion of the density function of the sub-Markovian semi-group $P_tf(x) = \E_x[f(X_t)\1_{\tau_0 > t}]$, which is obtained from the serie expansion \eqref{expansion}, written further. However, integrating this expansion over a probability measure satisfying \eqref{integrability} (for example a probability measure admitting a density function decaying as $x \mapsto e^{-r' x}$, with $r' > r$), Fubini's theorem seems not to be well justified, so that the asymptotic expansion of the density function for a general initial distribution satisfying \eqref{integrability} is not obvious.      
\end{remark} 

\subsection{Asymptotic behavior for the Bessel-3 process}
We will now proceed to the proof of Theorem \ref{theorem}.
 
To do so, the main strategy is to use the $Q$-process as a Doob transform for the sub-Markovian semi-group $P_tf(x) = \E_x[f(X_t)\1_{\tau_0 > t}]$. In particular, it is well-known that the asymptotic behavior of this Doob transform is very linked to the one of the conditional probability measure $\P_\mu[X_t \in \cdot | \tau_0 > t]$, for some $\mu \in \cM_1((0,+\infty))$, as it was shown for example in \cite{DM2015, diaconis2019analytic, ocafrain2020} in the context of absorbed Markov processes, or in \cite{ferre2018more} for exploding Feynman-Kac semi-groups.

Hence, before proving Theorem \ref{theorem}, the following lemma will be first stated and proved : 
\begin{lemma}
\label{lemma}
For any measurable function $f$ such that $$C_f := \frac{1}{2} \int_{\R_+} |f(x)| x^2 dx < \infty~~\text{ and }~~C'_f := \frac{1}{2} \int_{\R_+} |f(x)| x^4 dx < \infty,$$ for any $t \geq 0$ and for any probability measure $\mu$ supported on $(0,+\infty)$ satisfying $\int_0^\infty x^{2} \mu(dx) < + \infty$, it holds
$$ |\gamma(f)-K_t \mu Q_tf| \leq \frac{C_f \int_0^\infty x^2 \mu(dx) + C'_f}{t},$$
where $$K_t := \frac{t \sqrt{2 \pi t}}{2}.$$
If moreover $f$ is positive, 
\begin{equation}
\label{approx}
t \times (\gamma(f)-K_t \mu Q_tf) \underset{t \to \infty}{\longrightarrow} \int_0^\infty \int_0^\infty f(y) \frac{y^2(x^2+y^2)}{2}dy \mu(dx).
\end{equation}
\end{lemma}
\begin{proof}
Let $f$ satisfying $\int_{\R_+} |f(x)|x^4dx < + \infty$ and $\int_{\R_+} |f(x)|x^2dx < + \infty$, then let $x > 0$. By the explicit formula of the density function \eqref{density},
$$K_t Q_tf(x) = \int_{\R_+} f(y) t \frac{y}{x} \sinh\left(\frac{xy}{t}\right) e^{-\frac{x^2+y^2}{2t}}dy.$$
As a result, for any $t \geq 0$,
\begin{align*}K_t Q_tf(x) - \gamma(f) &= \int_{\R_+} f(y) \left[t \frac{y}{x} \sinh\left(\frac{xy}{t}\right) e^{-\frac{x^2+y^2}{2t}} - y^2\right]dy \\
&= \int_{\R_+} f(y) t \frac{y}{x}\left[ \sinh\left(\frac{xy}{t}\right) e^{-\frac{x^2+y^2}{2t}} - \frac{xy}{t}\right]dy
\end{align*}
Defining $g : z \mapsto e^{-z} - 1 + z$, one has 
\begin{align}\sinh\left(\frac{xy}{t}\right) e^{-\frac{x^2+y^2}{2t}} - \frac{xy}{t} &= \frac{1}{2} \left(g\left[\frac{(x-y)^2}{2t}\right] - g\left[\frac{(x+y)^2}{2t}\right]\right) \notag \\
&= -\frac{1}{2} \int_{\frac{(x-y)^2}{2t}}^{\frac{(x+y)^2}{2t}} g'(z)dz, &\forall y > 0, \forall t \geq 0. \label{p}\end{align}
However, denoting $h : z \mapsto \frac{z^2}{2}$, using that $g'(z) = 1-e^{-z} \leq z = h'(z)$ for any $z \geq 0$,
$$ \int_{\frac{(x-y)^2}{2t}}^{\frac{(x+y)^2}{2t}} g'(z)dz \leq  \int_{\frac{(x-y)^2}{2t}}^{\frac{(x+y)^2}{2t}} h'(z)dz = \frac{(x+y)^4-(x-y)^4}{8t^2} = \frac{xy(x^2+y^2)}{t^2}. $$
As a result, for any $y \in \R_+$ and $t \geq 0$,
$$\left|\sinh\left(\frac{xy}{t}\right) e^{-\frac{x^2+y^2}{2t}} - \frac{xy}{t}\right| \leq \frac{xy(x^2+y^2)}{2t^2}.$$
Thus, for any $t \geq 0$,
\begin{align}
|K_t Q_tf(x) - \gamma(f)| &\leq \int_{\R_+} |f(y)| \frac{y^2(x^2+y^2)}{2t}dy \notag \\ 
&\leq \frac{C_f x^2 + C'_f }{t}. \label{name}\end{align}
As a result, integrating over a probability measure $\mu(dx)$ supported on $(0,+\infty)$,
\begin{align}
|K_t \mu Q_tf - \gamma(f)| 
&\leq \frac{C_f \int_0^\infty x^2 \mu(dx) + C'_f}{t} \label{lebesgue},\end{align}
which is the first part of the lemma.
Now, assume moreover that $f$ is positive. Then, since $\frac{xy}{t} - \sinh\left(\frac{xy}{t}\right)e^{-\frac{x^2+y^2}{2t}} \geq 0$ for any $x,y > 0$ (this is proved by \eqref{p}), 
$$\left| K_t Q_tf(x) - \gamma(f)\right| = \int_0^\infty f(y) t \frac{y}{x} \left(\frac{xy}{t} - \sinh\left(\frac{xy}{t}\right)e^{-\frac{x^2+y^2}{2t}}\right)dy,$$
and for any probability measure $\mu$ supported on $(0,+\infty)$,
$$\gamma(f) - K_t \mu Q_tf = \int_0^\infty \int_0^\infty f(y) t \frac{y}{x} \left(\frac{xy}{t} - \sinh\left(\frac{xy}{t}\right)e^{-\frac{x^2+y^2}{2t}}\right)dy\mu(dx).$$
The function $y \mapsto  \sinh\left(\frac{xy}{t}\right)e^{-\frac{x^2+y^2}{2t}}$ can be expressed as a serie expansion and one has
\begin{equation} \label{expansion}\frac{xy}{t} - \sinh\left(\frac{xy}{t}\right)e^{-\frac{x^2+y^2}{2t}} = \frac{xy(x^2+y^2)}{2t^2} + \sum_{n \geq 3} \frac{a_n(x,y)}{t^n},\end{equation}
 where, for any $n \geq 3$,
$$a_n(x,y) := \frac{(-1)^n}{2^{n} n!}\left((x+y)^{2n} - (x-y)^{2n}\right).$$
Hence, for any $x , y > 0$,
$$ t^2 \frac{y}{x} \left(\frac{xy}{t} - \sinh\left(\frac{xy}{t}\right)e^{-\frac{x^2+y^2}{2t}}\right) \underset{t \to \infty}{\longrightarrow} \frac{y^2(x^2+y^2)}{2}.$$
 Thus, using \eqref{name}, by Lebesgue's theorem, one shows that if $\mu$ is a probability measure supported on $(0,\infty)$ satisfying $\int_0^\infty x^{2} \mu(dx) < + \infty$, then one has
\begin{equation*}
\label{approx2}
t \times (\gamma(f)-K_t \mu Q_tf ) \underset{t \to \infty}{\longrightarrow} \int_0^\infty \int_0^\infty f(y) \frac{y^2(x^2+y^2)}{2}dy \mu(dx).
\end{equation*}
\end{proof}

\subsection{Proof of Theorem \ref{theorem}}
Theorem \ref{theorem} will now be proved. 
Let $\mu$ be a probability measure supported on $(0,+\infty)$ satisfying
$$\int_0^\infty x^2 \eta(x) \mu(dx) < + \infty.$$
Remark that this above-mentioned condition is exactly the condition \eqref{integrability}. Also remark that this condition implies that
$$\int_0^\infty \eta(x) \mu(dx) < + \infty.$$

The first step is to prove that there exists $t_\mu$ and $C_\mu < + \infty$ such that, for any $t \geq t_\mu$,
$$\cW_1(\P_\mu(X_t \in \cdot | \tau_0 > t), \alpha_{Yaglom}) \leq C_\mu/t,$$
which will imply that $\limsup_{t \to \infty} t \times \cW_1(\P_\mu(X_t \in \cdot | \tau_0 > t), \alpha_{Yaglom}) < + \infty$.

In what follows, one will use the following notation
$$\eta \circ \mu(dx) := \frac{\eta(x) \mu(dx)}{\mu(\eta)}.$$
Then, for any $t \geq 0$ and $f \in \Lip_{1,1}(\R_+)$,
\begin{align}
\E_\mu[f(X_t) | \tau_0 > t] &= \frac{\int_0^\infty \E_x[f(X_t) \1_{\tau_0 > t}] \mu(dx)}{\int_0^\infty \E_x[\1_{\tau_0 > t}] \mu(dx)} \notag \\
&=  \frac{K_t \int_0^\infty \frac{e^{\lambda_0 t}}{\eta(x)}\E_x[\eta(X_t)\frac{f(X_t)}{\eta(X_t)} \1_{\tau_0 > t}] \eta(x) \mu(dx)}{K_t \int_0^\infty \frac{e^{\lambda_0 t}}{\eta(x)}\E_x[\eta(X_t)\frac{1}{\eta(X_t)} \1_{\tau_0 > t}] \eta(x)\mu(dx)} \notag \\
&= \frac{K_t \int_0^\infty Q_t[f/\eta](x) \eta(x) \mu(dx)}{K_t \int_0^\infty Q_t[\1_{\R_+}/\eta](x) \eta(x)\mu(dx)} \notag \\
&= \frac{K_t (\eta \circ \mu) Q_t [f/\eta]}{K_t (\eta \circ \mu) Q_t [\1_{\R_+}/\eta]}, \label{ef2}
\end{align}
where we recall that the semi-group $(Q_t)_{t \geq 0}$ was defined previously in \eqref{q-proc}.
Then, let us note that
$$\sup_{f \in \Lip_{1,1}(\R_+)} C_{f/\eta} =  \sup_{f \in \Lip_{1,1}(\R_+)} \int_0^\infty |f|(x) x e^{-rx}dx \leq  \int_0^\infty (1+x) x e^{-rx}dx =: C < + \infty$$
and 
$$\sup_{f \in \Lip_{1,1}(\R_+)} C'_{f/\eta} = \sup_{f \in \Lip_{1,1}(\R_+)} \int_0^\infty |f|(x) x^3 e^{-rx}dx \leq  \int_0^\infty (1+x) x^3 e^{-rx}dx =: C' < + \infty.$$
Thus, by Lemma \ref{lemma} and noting that $\gamma(f/\eta) = \alpha_{\Yaglom}(f)$, for any $f \in \Lip_{1,1}(\R_+)$,
$$|K_t (\eta \circ \mu) Q_t[f/\eta] - \alpha_\Yaglom(f)| \leq \frac{C \int_0^\infty x^2 (\eta \circ \mu)(dx) + C'}{t}.$$
Since $\1_{\R_+} \in \Lip_{1,1}(\R_+)$, one has also
\begin{equation} \label{un}|K_t (\eta \circ \mu) Q_t(\1_{\R_+}/\eta) - 1| \leq \frac{C \int_0^\infty x^2 (\eta \circ \mu)(dx) + C'}{t}.\end{equation}
Therefore, given \eqref{ef2}, for any $t > C_\mu := C \int_0^\infty x^2 (\eta \circ \mu)(dx) + C'$, 
\begin{equation}
\label{encadrement}
    \frac{\alpha_\Yaglom(f) - \frac{C_\mu}{t}}{1 + \frac{C_\mu}{t}} \leq \E_\mu(f(X_t)|\tau_0 > t) \leq  \frac{\alpha_\Yaglom(f) + \frac{C_\mu}{t}}{1 - \frac{C_\mu}{t}}.
\end{equation}
Using that $|\alpha_\Yaglom(f)| \leq 1 + 2/r$ for any $f \in \Lip_{1,1}(\R_+)$, for any $t \geq C_\mu +1$,
\begin{align*}
\frac{\alpha_\Yaglom(f) + \frac{C_\mu}{t}}{1 - \frac{C_\mu}{t}} &= \left(\alpha_\Yaglom(f) + \frac{C_\mu}{t}\right)\left(1 + \frac{\frac{C_\mu}{t}}{1 - \frac{C_\mu}{t}}\right) \\
&\leq \alpha_\Yaglom(f) + \frac{C_\mu}{t} +  \left(1+2/r + \frac{C_\mu}{C_\mu+1}\right) \frac{ \frac{C_\mu}{t}}{1 - \frac{C_\mu}{t}} \\
&\leq \alpha_\Yaglom(f) + \frac{C'_\mu}{t},
\end{align*}
where
$$C'_\mu := C_\mu \left(1+\frac{1+2/r + \frac{C_\mu}{C_\mu+1}}{1-\frac{C_\mu}{1+C_\mu}}\right).$$
In a same way, one can prove that, for any $t \geq C_\mu + 1$,
$$\alpha_\Yaglom(f) - \frac{C''_\mu}{t} \leq \frac{\alpha_\Yaglom(f) - \frac{C_\mu}{t}}{1 + \frac{C_\mu}{t}}$$
with
$$C''_\mu := C_\mu \left(2+2/r + \frac{C_\mu}{C_\mu+1}\right).$$
As a result, using \eqref{encadrement}, for any $t \geq C_\mu + 1$,
$$\cW_1(\P_\mu[X_t \in \cdot | \tau_0 > t], \alpha_\Yaglom) \leq \frac{C'_\mu \lor C''_\mu}{t},$$
which concludes the first step.

Now, set $f : x \mapsto 1+x$. Then $f \in \Lip_{1,1}(\R_+)$, positive on $\R_+$, and one has therefore $$\cW_1(\P_\mu(X_t \in \cdot | \tau_0 > t), \alpha_{\Yaglom}) \geq | \E_\mu(f(X_t)| \tau_0 > t) - \alpha_\Yaglom(f)|.$$
 Moreover, by the computation of the moments of $\alpha_\Yaglom$, $f$ satisfies the following inequality : \begin{equation}\label{cond}\int_0^\infty y^2 f(y) \alpha_\Yaglom(dy) > \alpha_\Yaglom(f) \int_0^\infty y^2 \alpha_\Yaglom(dy).\end{equation}
Denote $\psi_{\mu}$ the positive measure defined by \begin{equation}\label{notation}\psi_{\mu} : f \mapsto \int_0^\infty \int_0^\infty f(y) \frac{y^2(x^2+y^2)}{2}\frac{\eta(x)}{\eta(y)}\frac{\mu(dx)}{\mu(\eta)}dy.\end{equation} 
Then, since $\frac{y^2}{\eta(y)}dy = \alpha_\Yaglom(dy)$, one has
\begin{align}
\psi_{\mu}(f) &=  \int_0^\infty \int_0^\infty f(y) \frac{(x^2+y^2)}{2} \eta(x)\frac{\mu(dx)}{\mu(\eta)}\alpha_\Yaglom(dy) \notag \\ & = \frac{1}{2}\alpha_\Yaglom(f) \int_0^\infty x^2 (\eta \circ \mu)(dx) + \frac{1}{2}\int_0^\infty f(y)y^2 \alpha_\Yaglom(dy), \label{suite}\end{align}
and the condition \eqref{cond} implies therefore that
$$\psi_{\mu}(f) > \alpha_\Yaglom(f) \psi_{\mu}(\1_{\R_+}).$$
Then, by the equation \eqref{ef2} and using the second part of Lemma \ref{lemma}, for any $t \geq 0$,
\begin{align*}
\E_\mu(f(X_t) | \tau_0 > t) &= \frac{K_t (\eta \circ \mu)Q_t[f/\eta]}{K_t (\eta \circ \mu)Q_t[\1_{\R_+}/\eta]} \\
&= \frac{\alpha_\Yaglom(f) - \frac{1}{t} \psi_{\mu}(f) + o(1/t)}{1 - \frac{1}{t} \psi_{\mu}(\1_{\R_+}) + o(1/t)}  \\
&= \left[ \alpha_\Yaglom(f) - \frac{1}{t}\psi_{\mu}(f) + o(1/t)\right] \left[1 + \frac{1}{t}\psi_{\mu}(\1_{\R_+}) + o(1/t)\right] \\
&= \alpha_\Yaglom(f) + \frac{1}{t}\left (\alpha_\Yaglom(f) \psi_{\mu}(\1_{\R_+}) - \psi_{\mu}(f)\right) + o(1/t).
\end{align*}
So, for any $\epsilon \in (0,\psi_{\mu}(f) - \alpha_\Yaglom(f) \psi_{\mu}(\1_{\R_+}))$ and for $t$ large enough,
$$t \times |\E_\mu(f(X_t) | \tau_0 > t) - \alpha_{\Yaglom}(f)| \geq \psi_{\mu}(f) - \alpha_\Yaglom(f) \psi_{\mu}(\1_{\R_+}) - \epsilon > 0,$$
which proves that
\begin{multline*}\liminf_{t \to \infty}~ t \times \cW_1(\P_\mu(X_t \in \cdot | \tau_0 > t), \alpha_{\Yaglom}) \\ \geq \liminf_{t \to \infty} ~t \times | \E_\mu(f(X_t)| \tau_0 > t) - \alpha_\Yaglom(f)| > 0.\end{multline*}
This ends the proof. 

\section{Polynomial convergence to the Bessel-$3$ process}

Now, let us state the following theorem:

\begin{theorem}
\label{q-processus}
There exists $s_0 > 0$ such that, for any $\mu \in \cM_1((0,+\infty))$ satisfying $\int_0^\infty x^4 e^{rx} \mu(dx) < + \infty$ and for any $s \geq s_0$,
\begin{multline*}
0 < \liminf_{t \to \infty} ~t \times \cW_1(\P_\mu[X_{s} \in \cdot | \tau_0 > t], \Q_{\eta \circ \mu}[X_{s} \in \cdot]) \\ \leq  \limsup_{t \to \infty}~ t \times\cW_1(\P_\mu[X_{s} \in \cdot | \tau_0 > t],\Q_{\eta \circ \mu}[X_{s} \in \cdot]) < +\infty.\end{multline*}
\end{theorem}
\begin{remark}
Note that, in this theorem, the assumption of integrability on the initial measure is slightly stronger than \eqref{integrability}. This is due to the use of the $1$-Wasserstein distance. For the total variation distance, the condition \eqref{integrability} is suitable to obtain the same statement as Theorem \ref{q-processus}.  \end{remark}
Before proving this theorem, the following lemma is needed:
\begin{lemma}
\label{lamme}
If $\mu$ is a probability measure such that $\int_0^\infty x^3 e^{rx} \mu(dx) < + \infty$, then, for any $s \geq 0$, 
\begin{equation}
    \label{aim}
    \E_\mu[X_s^3 e^{rX_s} | \tau_0 > s] < + \infty.
\end{equation}
\end{lemma}
\begin{proof}
Let $s \geq 0$. Then, by \eqref{q-proc}, it is enough to prove that
$\E^\Q_{\eta \circ \mu}[X_s^2] < + \infty$.
For any $x > 0$, under $\Q_x$, $(X_t)_{t \geq 0}$ is a Bessel-3 process following 
$$dX_t = dW_t + \frac{1}{X_t}dt,$$
given a one-dimensional Brownian motion $(W_t)_{t \geq 0}$. 
Let $x > 0$ and, for any $N > x$, denote $T_N := \inf\{t \geq 0 : X_t = N\}$. Then, by Itô's formula applied to $s \land T_N$, for any $N > x$, 
$$\E^\Q_x[X_{s \land T_N}^2] = x^2 + \E^\Q_x\left[\int_0^{s \land T_N} 2 X_u dW_u\right] + 3 \E^\Q_x[s \land T_N].$$ 
Since $\E^\Q_x\left[\int_0^{s \land T_N} X^2_u du\right] < + \infty$ for any $s \geq 0$, the process $(\int_0^{s \land T_N} 2 X_u dW_u)_{s \geq 0}$ is a martingale under $\Q_x$, so $\E^\Q_x\left[\int_0^{s \land T_N} 2 X_u dW_u\right] = 0$ for any $N > x$. As a result,
$$\E^\Q_x[X_{s \land T_N}^2] = x^2 + 3 \E^\Q_x[s \land T_N],~~~~\forall N > x.$$ 
Now, by Fatou's lemma and the monotone convergence theorem, for any $s \geq 0$,
\begin{equation}
\label{fatou}\E^\Q_x[X_{s}^2] \leq \liminf_{N \to + \infty} \E^\Q_x[X_{s \land T_N}^2] = x^2 + 3 \liminf_{N \to \infty} \E^\Q_x[s \land T_N] = x^2 + 3s.\end{equation}
Hence, for any probability measure $\mu$ such that  $\int_0^\infty x^3 e^{rx} \mu(dx) < + \infty$,
$$\E^\Q_{\eta \circ \mu}[X_s^2] = \int_0^\infty \frac{\eta(x) \mu(dx)}{\mu(\eta)} \E^\Q_x[X_s^2] \leq  \int_0^\infty \frac{\eta(x) \mu(dx)}{\mu(\eta)} (x^2 + 3s) < + \infty.$$
\end{proof}
Now, let us prove Theorem \ref{q-proc}:
\begin{proof}[Proof of Theorem \ref{q-proc}]
Let $\mu$ be a probability measure supported on $(0,+\infty)$ satisfying $\int_0^\infty x^4 e^{rx} \mu(dx) < + \infty$ and $s \leq t$. Then, by Markov property and using the notation $\phi_s(\mu) := \P_\mu(X_s \in \cdot | \tau_0 > s)$, for any $f \in \Lip_{0,1}(\R_+)$,
\begin{align}
\label{computation1} 
\E_\mu[f(X_{s}) | \tau_0 > t]  = \frac{\E_\mu[f(X_s)\1_{\tau_0 > s} \P_{X_s}[\tau_0 > t-s]]}{\E_\mu[\1_{\tau_0 > s} \P_{X_s}[\tau_0 > t-s]]} &= \E_\mu\left[\frac{{f(X_{s}) \1_{\tau_0 > s}}}{\P_\mu[\tau_0 >s]} \frac{\P_{X_s}[\tau_0 > t-s]}{\P_{\phi_s(\mu)}[\tau_0 > t-s]}\right].
\end{align}
Then, by the definition of $(Q_t)_{t \geq 0}$ \eqref{q-proc},
\begin{align}
\label{computation2}
\frac{\P_{X_s}[\tau_0 > t-s]}{\P_{\phi_s(\mu)}[\tau_0 > t-s]} 
&=  \frac{\eta(X_s)}{\phi_s(\mu)(\eta)}  \frac{K_{t-s}Q_{t-s}[\1_{\R_+}/\eta](X_s)}{K_{t-s}(\eta \circ \phi_s(\mu))Q_{t-s}[\1_{\R_+}/\eta]}.
\end{align}
By Lemma \ref{lamme}, $C_{\phi_s(\mu)} < + \infty$ (recalling that $C_\mu$ is defined in the proof of Theorem \ref{theorem}, for any $\mu$). Then, by \eqref{un}, the following inequalities hold for $t \geq s + C_{\phi_s(\mu)}$:
\begin{align*}
\frac{1 - \frac{C_{X_s}}{t-s}}{1 + \frac{C_{\phi_s(\mu)}}{t-s}} &\leq  \frac{K_{t-s}Q_{t-s}[\1_{\R_+}/\eta](X_s)}{K_{t-s}(\eta \circ \phi_s(\mu))Q_{t-s}[\1_{\R_+}/\eta]} \leq \frac{1 + \frac{C_{X_s}}{t-s}}{1 - \frac{C_{\phi_s(\mu)}}{t-s}}.
\end{align*}
As a result, for any $f \in \Lip_{0,1}(\R_+)$ and for any $t \geq s + C_{\phi_s(\mu)}$,
\begin{align}\left| \E_\mu[f(X_{s})|\tau_0 > t] - \E^\Q_{\eta \circ \mu}[f(X_{s})]\right| &\leq  \E_\mu\left[\frac{{X_s \1_{\tau_0 > s}}}{\P_\mu[\tau_0 >s]} \frac{\eta(X_s)}{\phi_s(\mu)(\eta)}\left( \frac{C_{\phi_s(\mu)} + C_{X_s}}{t - s -  C_{\phi_s(\mu)}}\right)\right] \notag \\
&= \frac{ (C_{\phi_s(\mu)} + C') \E_{\eta \circ \mu}^\Q[X_s] + C \E^\Q_{\eta \circ \mu}[X_{s}^3]}{t-s-C_{\phi_s(\mu)}}. \label{lebesgue}\end{align}
Mimicking the proof of Lemma \ref{lamme}, one can show that, for any $x > 0$, 
$$\E^\Q_x[X^3_{s}] \leq x^3 + 6 \E^\Q_x\left[\int_0^{s} X_udu\right] \leq x^3 + 6 \int_0^s \sqrt{x^2 + 3u}du = x^3 + \frac{4}{3} \left[(x^2 + 3s)^{3/2} - x^3 \right],$$
where the second inequality is due to a Cauchy-Schwartz's inequality and \eqref{fatou}.
Hence, since $\int_0^\infty x^4 e^{rx}\mu(dx) < + \infty$, 
$$\E^\Q_{\eta \circ \mu}[X_{s}^3] = \int_0^\infty (\eta \circ \mu)(dx) \E^\Q_{x}[X_{s}^3]  < + \infty.$$
Thus $\limsup_{t \to \infty}~ t \times\cW_1(\P_\mu[X_{s} \in \cdot | \tau_0 > t], \Q_{\eta \circ \mu}[X_{s} \in \cdot]) < +\infty$. Now, set $g : x \mapsto (1-x) \lor 0 \in \Lip_{1,1}(\R_+)$. By the previous computations \eqref{computation1}, \eqref{computation2}, and by \eqref{q-proc}, one has 
\begin{multline*}\E_\mu[g(X_{s})|\tau_0 > t] - \E^\Q_{\eta \circ \mu}[g(X_{s})] \\ =  \E_\mu\left[\frac{{g(X_{s}) \1_{\tau_0 > s}}}{\P_\mu[\tau_0 >s]} \frac{\eta(X_s)}{\phi_s(\mu)(\eta)}\left( \frac{K_{t-s}Q_{t-s}[\1_{\R_+}/\eta](X_s) - K_{t-s}(\eta \circ \phi_s(\mu))Q_{t-s}[\1_{\R_+}/\eta]}{K_{t-s}(\eta \circ \phi_s(\mu))Q_{t-s}[\1_{\R_+}/\eta]}\right)\right] \\ =\frac{\E^\Q_{\eta \circ \mu}\left[g(X_s) (K_{t-s}Q_{t-s}[\1_{\R_+}/\eta](X_s)- K_{t-s}(\eta \circ \phi_s(\mu))Q_{t-s}[\1_{\R_+}/\eta])\right]}{ K_{t-s}(\eta \circ \phi_s(\mu))Q_{t-s}[\1_{\R_+}/\eta]} 
.\end{multline*}
Using again the notation $\psi_\mu$ in \eqref{notation}, by what we showed in the proof of theorem \ref{theorem},
$$K_{t-s}(\eta \circ \phi_s(\mu))Q_{t-s}[\1_{\R_+}/\eta] = 1 - \frac{1}{t-s} \psi_{\phi_s(\mu)}(\1_{\R_+}) + o\left(\frac{1}{t-s}\right).$$  
Likewise, for any $y > 0$, and using \eqref{suite} for the third line,
\begin{multline*} K_{t-s}Q_{t-s}[\1_{\R_+}/\eta](y) - K_{t-s}(\eta \circ \phi_s(\mu))Q_{t-s}[\1_{\R_+}/\eta] \\=  \frac{1}{t-s} [\psi_{\phi_s(\mu)}(\1_{\R_+}) - \psi_{\delta_{y}}(\1_{\R_+})]+ o\left(\frac{1}{t-s}\right) \\ = \frac{1}{2(t-s)} \left[\int_0^\infty x^2 (\eta \circ \phi_s(\mu)) (dx) -  y^2 \right]+ o\left(\frac{1}{t-s}\right) \\= \frac{1}{2(t-s)} \left[\E^\Q_{\eta \circ \mu}[X^2_s] -  y^2 \right]+ o\left(\frac{1}{t-s}\right),\end{multline*}
where the last line is obtained from $\mu P_s \eta = e^{- \lambda_0 s} \mu(\eta)$, which implies that
$$\int_0^\infty x^2 (\eta \circ \phi_s(\mu))(dx) = \frac{\E_\mu[X_s^2 \eta(X_s) | \tau_0 > s]}{\E_\mu[\eta(X_s) | \tau_0 > s]} = \frac{\E_\mu[X_s^2 \eta(X_s) \1_{\tau_0 > s}]}{e^{- \lambda_0 s} \mu(\eta)} = \E_{\eta \circ \mu}^\Q[X_s^2].$$
Then, by Lebesgue's theorem (using \eqref{lebesgue}), one has
$$\lim_{t \to + \infty}2t \times \left(\E_\mu[g(X_{s})|\tau_0 > t] - \E^\Q_{\eta \circ \mu}[g(X_{s})]\right) = \E^\Q_{\eta \circ \mu}[g(X_s)]  \E^\Q_{\eta \circ \mu}[X_s^2]- \E^\Q_{\eta \circ \mu}[g(X_s)X_s^2]. $$
Since $g(X_s) = 0$ if and only if $X_s \geq 1$, one has 
$$\E_{\eta \circ \mu}^\Q[g(X_s)X_s^2] \leq \E_{\eta \circ \mu}^\Q[g(X_s)].$$
Since $X_t \underset{t \to + \infty}{\longrightarrow} + \infty$ $\Q_{0}$-almost surely, there exists $s_0 > 0$ such that, for any $s \geq s_0$ and $\mu \in \cM_1((0,+\infty))$, $\E^\Q_{\eta \circ \mu}[X_s^2] > 1$. Then, claiming moreover that $\E_{\eta \circ \mu}^\Q[g(X_s)] > 0$ for any $\mu \in \cM_1((0,+\infty))$ satisfying $\int_0^\infty x^4e^{rx} \mu(dx) < + \infty$ and $s \geq s_0$, one obtains
$$\E^\Q_{\eta \circ \mu}[g(X_s)X_s^2] < \E^\Q_{\eta \circ \mu}[g(X_s)] \E^\Q_{\eta \circ \mu}[X^2_s].$$  
This allows to conclude that $\liminf_{t \to + \infty} t \times \cW_1(\P_{\mu}[X_s \in \cdot | \tau_0 > t], (\eta \circ \mu) Q_s) > 0$. 
\end{proof}

\bibliographystyle{abbrv}
\bibliography{biblio-william}



\ACKNO{I am very thankful to the anonymous referees for the relevant remarks that they did for the improvement of this note.}


\end{document}